\documentclass[12pt]{amsart}
\usepackage[margin=1in]{geometry}
\usepackage{graphicx} 
\usepackage{amsthm}
\usepackage{amssymb}
\usepackage{amsmath}
\usepackage{comment}
\usepackage{tikz, pgfplots} 
\pgfplotsset{compat=1.18}
\usepackage{ytableau}
\usepackage{enumitem}
\usepackage{soul}
\usepackage{cite}
\usepackage{hyperref}

\newtheorem{thm}{Theorem}[section]
\newtheorem{prop}[thm]{Proposition}
\newtheorem{cor}[thm]{Corollary}
\newtheorem{lemma}[thm]{Lemma}
\theoremstyle{definition}
\newtheorem{example}[thm]{Example}
\newtheorem{defn}[thm]{Definition}

\DeclareMathOperator{\sign}{sign}
\DeclareMathOperator{\asc}{asc}
\DeclareMathOperator{\sh}{sh}
\DeclareMathOperator{\wt}{wt}

\definecolor{lblue}{RGB}{43, 188, 237}
\definecolor{dgreen}{RGB}{0, 166, 11}
\definecolor{dbrown}{RGB}{112, 79, 0}

\title{Generalized Hockey Stick Theorem}
\author{
  Molly Lynch\\
  Department of Mathematics, Statistics, and Computer Science\\
  Hollins University\\
  \texttt{lynchme2@hollins.edu}
  \vspace{1em} 
  \\
  Michael Weselcouch\\
  Department of Mathematics, Computer Science, and Physics\\
  Roanoke College\\
  \texttt{weselcouch@roanoke.edu}
}
\date{\today}

\begin{document}

\begin{abstract}
We give a combinatorial proof via a sign-reversing involution for a new identity that generalizes both the Hockey Stick Identity and the Big Hockey Stick and Pucks Identity.
\end{abstract}

\maketitle
\markleft{\MakeUppercase{\scshape Molly Lynch and Michael Weselcouch}}

\section{Introduction}

In this article we will give a combinatorial proof using a sign-reversing involution of the following identity:

\begin{equation}
\sum_{i=0}^k\binom{(r+1)i+n}{i} = \sum_{\alpha \in \mathcal{C}(k, r)} \sign(\alpha) \Biggl(\prod_{j=1}^{\ell(\alpha)}\binom{r}{\alpha_j-1}\Biggr)\binom{(r+1)k+n+1-|\alpha|+\ell(\alpha)}{k-|\alpha|}.
\tag{\ref{ZC}}
\end{equation}

Before introducing the notation used in the right hand side of the identity, we should note that when $r=0$, Equation (\ref{ZC}) gives the Hockey Stick Identity (also known as Chu's Theorem) \cite{Merris_2003} and when $r=1$, Equation (\ref{ZC}) gives the Big Hockey Stick and Pucks Identity \cite{Hilton_1987, MehriARXIV}.  For fixed values of $r$ and $n$, Equation \ref{ZC} gives many sequences from the On-Line Encyclopedia of Integer Sequences (OEIS \cite{OEIS}) as detailed in Table \ref{OEIS table}.

The Hockey Stick Identity is a well-known result in enumerative combinatorics.  It states that for nonnegative integers $n$ and $k$,
\begin{equation}\label{Hockey Stick Identity}
    \sum_{i=0}^k\binom{i+n}{i} = \binom{k+1+n}{k}.
\end{equation}
The reason behind naming this identity the ``Hockey Stick Identity" is because the terms in the identity appear to form a hockey stick when viewed as being a part of Pascal's triangle (\href{https://oeis.org/A007318}{A007318}).  Figure \ref{fig:HSI} shows the terms in the Hockey Stick Identity when $n=2$ and $k = 3$.  The terms that appear on the left-hand side of the Equation \ref{Hockey Stick Identity} are circled in brown and the term on the right-hand side is circled in gray.  In this case, we see that $\binom{2}{0} + \binom{3}{1} + \binom{4}{2} + \binom{5}{3} = 20 = \binom{6}{3}$.

\begin{figure}
    \centering
    \begin{tikzpicture}
\foreach \i in {0, ..., 3}{
\draw[color=dbrown, ultra thick, fill=brown!35] (\i/2-1,-\i-2) circle [radius=.4];
}
\draw[color=gray, ultra thick, fill = gray!35] (0,-6) circle [radius=.4];
\foreach \n in {0,...,6} {
  \foreach \k in {0,...,\n} {
    \node at (\k-\n/2,-\n) {$\binom{\n}{\k}$};
    }
    }
\end{tikzpicture}

    \caption{The Hockey Stick Identity with $n=2$ and $k = 3$.}
    \label{fig:HSI}
\end{figure}

Similarly, the Big Hockey Stick and Pucks Identity states that for nonnegative integers $n$ and $k$, 
\begin{equation}\label{Big Hockey Stick}
\sum_{i=0}^k\binom{2i+n}{i} = \sum_{i=0}^{\lfloor\frac{k}{2}\rfloor}(-1)^i\binom{2k+1+n-i}{k-2i}.
\end{equation}
Figure \ref{fig:HSPI} shows the terms in the Big Hockey Stick and Pucks Identity when $n=1$ and $k = 4$. As with Figure \ref{fig:HSI}, in Figure \ref{fig:HSPI}, the terms that appear on the left-hand side of the Equation \ref{Big Hockey Stick} are circled in brown and the terms that appear on the right-hand side are circled in gray.  In this case, we see that $\binom{1}{0} + \binom{3}{1} + \binom{5}{2} + \binom{7}{3} + \binom{9}{4} = 175 = \binom{10}{4} - \binom{9}{2} + \binom{8}{0}$.

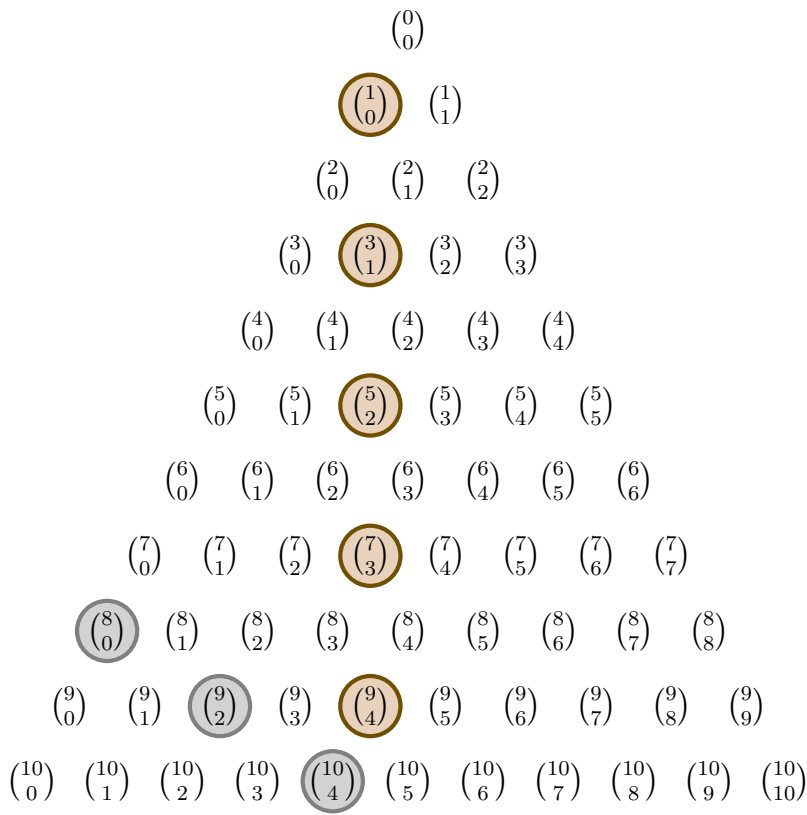
\begin{figure}
    \centering
\begin{tikzpicture}
\foreach \i in {0, ..., 4}{
\draw[color=dbrown, ultra thick, fill=brown!35] (-.5,-2*\i-1) circle [radius=.4];
}
\draw[color=gray, ultra thick, fill = gray!35] (-1,-10) circle [radius=.4];
\draw[color=gray, ultra thick, fill = gray!35] (-2.5,-9) circle [radius=.4];
\draw[color=gray, ultra thick, fill = gray!35] (-4,-8) circle [radius=.4];
\foreach \n in {0,...,10} {
  \foreach \k in {0,...,\n} {
    \node at (\k-\n/2,-\n) {$\binom{\n}{\k}$};
    }
    }
\end{tikzpicture}

    \caption{The Hockey Stick and Pucks Identity with $n=1$ and $k = 4$.}
    \label{fig:HSPI}
\end{figure}

\begin{table}
\begin{center}
\begin{tabular}{ |c|c|c| } 
 \hline
 $r$ & $n$ & OEIS sequence \\ 
 \hline 
 \hline
 0 & 0 & \href{https://oeis.org/A000027}{A000027} \\ 
 \hline
 0 & 1 & \href{https://oeis.org/A000217}{A000217} \\ 
 \hline
 0 & 2 & \href{https://oeis.org/A000292}{A000292} \\
 \hline
 0 & 3 & \href{https://oeis.org/A000332}{A000332} \\
 \hline
 0 & 4 & \href{https://oeis.org/A000389}{A000389}\\
 \hline 
 0 & 5 & \href{https://oeis.org/A000579}{A000579} \\ 
 \hline 
 1 & 0 & \href{https://oeis.org/A006134}{A006134} \\
 \hline 
 1 & 1 & \href{https://oeis.org/A079309}{A079309}\\
 \hline 
 1 & 2 & \href{https://oeis.org/A057552}{A057552}\\
 \hline
 1 & 3 & \href{https://oeis.org/A371965}{A371965} \\
 \hline
 1 & 4 & \href{https://oeis.org/A371964}{A371964} \\
 \hline
 1 & 5 & \href{https://oeis.org/A371963}{A371963} \\
 \hline
 2 & 0 & \href{https://oeis.org/A188675}{A188675}\\
 \hline
 2 & 1 & \href{https://oeis.org/A263134}{A263134} \\
 \hline
 2 & 2 & \href{https://oeis.org/A087413}{A087413}\\
 \hline
 3 & 0 & \href{https://oeis.org/A225612}{A225612} \\
 \hline
 4 & 0 & \href{https://oeis.org/A225615}{A225615} \\
 \hline
\end{tabular}
\caption{OEIS sequences for fixed values of $r$ and $n$}
\label{OEIS table}
\end{center}
\end{table}

We note that previous work generalizing the hockey stick identity can be found in \cite{Jones_1996}, and work more generally regarding Pascal's triangle in \cite{Bondarenko_1993}.

\section{Preliminaries}

We begin with some preliminaries about sign-reversing involutions, compositions, Young diagrams, and lattice paths. For more information, see \cite{Stanley}.

\subsection{Sign-Reversing Involutions}\label{SRI}

Our main tool in this paper is that of a sign-reversing involution. Let $X$ be a set of objects and to each object $x \in X$, assign a weight, $\wt(x)=\pm 1$. The usual objective is to calculate the total weight of the objects in $X$, namely $\sum_{x\in X} \wt(x)$. A \emph{sign-reversing involution} on $X$ is an involution $\tau:X \rightarrow X$ such that if $x \neq \tau(x)$, then $\wt(\tau(x)) = -\wt(x)$. In other words, $\tau$ pairs up elements of $X$ with opposite signs. We denote the fixed points of $\tau$, i.e. $x \in X$ such that $\tau(x)=x$ as Fix$(\tau)$. Therefore, \[\sum_{x\in X} \wt(x) = \sum_{x\in \text{Fix}(\tau)} \wt(x).\] Often it is the case that $\wt(x)=1$ for all $x \in \text{Fix}(\tau)$, so in that case, summing the weights of all elements in $X$ gives the total number of fixed points. 

\subsection{Compositions}\label{comps}

A \textit{composition} $\alpha= (\alpha_1, \alpha_2, \dots, \alpha_j)$ of $k$ is a finite sequence of positive integers summing to $k$.  We write $\alpha \vDash k$ if $\alpha$ is a composition of $k$.  Similarly, we write $|\alpha|$ for the sum of the parts of $\alpha$ and $\ell(\alpha)$ for the number of parts of $\alpha$. We define the \textit{sign of} $\alpha$ to be $\sign(\alpha)=(-1)^{|\alpha|-\ell(\alpha)}$.

Given a composition $\alpha = (\alpha_1, \alpha_2, \dots, \alpha_j)$, we define the compositions $\alpha^{\rightarrow}, \alpha^{\uparrow}$, and $\alpha^{\swarrow}$ by: 
\begin{itemize}[label = -]
\item $\alpha^{\rightarrow} = (\alpha_1+1, \alpha_2, \dots, \alpha_j)$,
\item $\alpha^{\uparrow} = (2, \alpha_1, \alpha_2, \dots, \alpha_j)$,
\item $\alpha^{\swarrow} =\begin{cases} 
       (\alpha_1-1, \alpha_2, \dots, \alpha_j)& \text{if } \alpha_1>2, \\
        (\alpha_2, \dots, \alpha_j)& \text{if } \alpha_1 \leq 2.

   \end{cases}$
\end{itemize}

We note that if $\alpha_1 \geq 2$, these operations on $\alpha$ each result in a composition with sign opposite of the sign of $\alpha$.

\begin{prop}\label{signofalpha}
    If $\alpha_1\geq 2$, then $\sign(\alpha^{\rightarrow}) = \sign(\alpha^{\uparrow}) = \sign(\alpha^{\swarrow}) =  -\sign(\alpha)$.
\end{prop}
\begin{proof}
We prove each claim one at a time. 
\begin{itemize}[label = -]
    \item $\sign(\alpha^{\rightarrow}) = -\sign(\alpha)$: In this case, we have $\ell(\alpha^{\rightarrow}) = \ell(\alpha)$ but $|\alpha^{\rightarrow}| = |\alpha|+1$. As a result, $|\alpha^{\rightarrow}|-\ell(\alpha^{\rightarrow}) = |\alpha| - \ell(\alpha) + 1$. It follows that $\sign(\alpha^{\rightarrow})=-\sign(\alpha).$
    \item $\sign(\alpha^{\uparrow}) = -\sign(\alpha)$: Here we have $\ell(\alpha^{\uparrow}) = \ell(\alpha)+1$ and $|\alpha^{\uparrow}|=|\alpha| + 2$. As a result, $|\alpha^{\uparrow}| - \ell(\alpha^{\uparrow}) = |\alpha| - \ell(\alpha) + 1$. This again implies that $\sign(\alpha^{\uparrow})=-\sign(\alpha).$
    \item $\sign(\alpha^{\swarrow})=-\sign(\alpha):$ If $\alpha_1 > 2,$ then $\ell(\alpha^{\swarrow}) = \ell(\alpha)$ and $|\alpha^{\swarrow}|=|\alpha|-1$. From this, we have $|\alpha^{\swarrow}|-\ell(\alpha) = |\alpha|-\ell(\alpha)-1$ resulting in $\sign(\alpha^{\swarrow})=-\sign(\alpha)$. Finally, if $\alpha_1 =2$, then $\ell(\alpha^{\swarrow}) = \ell(\alpha)-1$ and $|\alpha^{\swarrow}|=|\alpha|-2.$ As a result, $\sign(\alpha^{\swarrow})=-\sign(\alpha).$
    
\end{itemize}
\end{proof}

Let $c(k, r)$ denote the set of compositions  $\alpha \vDash k$ such that $2 \leq \alpha_i \leq r+1$ for all $1 \leq i \leq \ell(\alpha)$ and let $\mathcal{C}(k, r) =\displaystyle \bigcup_{j=0}^k c(j, r)$.  

\begin{example}
If $k = 5$ and $r = 3$, then 
\[\mathcal{C}(5, 3) = \{(), (2), (3), (4), (2, 2), (2, 3), (3, 2) \}.\]
\end{example}

\subsection{Young Diagrams}\label{YD}
For a composition $\alpha$, the Young diagram of shape $\alpha$ is a collection of boxes arranged in left-justified rows such that row $i$ has $\alpha_i$ boxes. The compositions $\alpha^{\rightarrow}, \alpha^{\uparrow}$, and $\alpha^{\swarrow}$ defined in Section \ref{comps} can be described in term of Young diagrams. The Young diagram for $\alpha^{\rightarrow}$ can be obtained from the Young diagram for $\alpha$ by adding an additional box in the first row. The 
Young diagram for $\alpha^{\uparrow}$ can be obtained from the Young diagram for $\alpha$ by creating a new first row with two boxes.  Finally, $\alpha^{\swarrow}$ is the Young diagram obtained from $\alpha$ by removing the first row in the case where $\alpha_i\leq 2,$ and otherwise removing a box from the first row.

We define a \emph{labeled tableau of shape $\alpha$} to be a labeling of the boxes of
the Young diagram of shape $\alpha$ such that the first entry in each row is empty and the remaining entries are nonnegative integers that are strictly increasing from left to right.  We will denote by $\Lambda(\alpha, r)$ the set of all labeled tableau of shape $\alpha$ whose labels are all less than $r$.  For $i, j \geq 1$, we will let $\lambda_i$ denote the set of entries of the $i$th row of $\lambda$, $|\lambda_i|$ denote the number of boxes in the $i$th row, and we will let $\lambda_{i, j}$ denote the entry in the $i$th row and $(j+1)$st column of $\lambda$.  Finally, if $\lambda \in \Lambda(\alpha, r)$, then we define the \emph{sign of $\lambda$} as $\sign(\lambda) = \sign(\alpha)$.

\begin{example}\label{Labeled Tableau}
The following is an element of $\Lambda((4, 5, 3), 4)$:
\begin{center}
$\begin{ytableau}
       {} & 0 & 2 & 3 \\
       {} & 0 & 1 & 2 & 3 \\
       {} & 1 & 3 \\
\end{ytableau}$
\end{center}
In this case, $\lambda_{1, 1} = 0$ since $0$ is in the first row and second column of $\lambda$. We also have that $\sign(\lambda)=-1$.
\end{example}

\begin{prop}\label{numberLT}
Let $r$ be a nonnegative integer. Then if $r=0$, the size of the set $\Lambda(\alpha, r)$ is 1.  If $r \geq 1$, then
\[|\Lambda(\alpha, r)| = \prod_{j=1}^{\ell(\alpha)}\binom{r}{\alpha_j-1}.\]
\end{prop}

\begin{proof}
If $r = 0$, then the only element of $\Lambda(\alpha, r)$ is the  labeled tableau with empty shape.

Now assume $r\geq 1$.  The constraints on the labels of the labeled tableaux in $\Lambda(\alpha, r)$ are that the first entry of each row must be empty, the remaining entries must be nonnegative integers listed in strictly increasing order from left to right, and the labels must be less than $r$.  That means that for all $1\leq j \leq \ell(\alpha)$, the labels used in row $j$ must be an $(\alpha_j-1)$-element subset of $\{0, \dots, r-1\}$.  There are $\binom{r}{\alpha_j-1}$ such subsets.  Since the labels in each row are independent from the labels in the other rows, taking the product over the number of rows gives the desired result.  
\end{proof}

Given $\lambda \in \Lambda(\alpha, r)$ and $a \in \{0, \dots, r-1\}$, we define $\lambda \oplus a$ to be the labeled tableau of shape $\alpha^\uparrow$ satisfying $(\lambda \oplus a)_{1, 1} = a$ and for $i \geq 1$, $(\lambda \oplus a)_{i+1} = \lambda_{i}$.  If $a \notin \lambda_1$, then we let $\lambda + a$ to be the labeled tableau of shape $\alpha^\rightarrow$ that satisfies  $(\lambda+a)_1 = \lambda_1 \cup \{a\}$ and for $i > 1,$ $(\lambda+a)_i = \lambda_i$.  Finally we let $\lambda^{\swarrow}$ be the labeled tableau of shape $\alpha^{\swarrow}$ where if $|\lambda_1| > 2$, then $(\lambda^{\swarrow})_1 = \lambda_1 - \lambda_{1, 1}$ and for $i > 1,$ $(\lambda^{\swarrow})_i = \lambda_i$, but if $|\lambda_1| = 2$, then for $i \geq 1$,  $(\lambda^{\swarrow})_i = \lambda_{i+1}$.
That is to say, $\lambda^{\swarrow}$ is created from $\lambda$ by removing the smallest entry in the first row of $\lambda$, if $\alpha_1 > 2$ or it is created by removing the entire first row of $\lambda$ otherwise.


\begin{example}
If $\lambda$ is the labeled tableau shown in Example \ref{Labeled Tableau}, then $\lambda\oplus1$, $\lambda + 1$, and $\lambda^{\swarrow}$ are the following labeled tableaux:
\begin{center}
$\lambda \oplus 1 = 
\begin{ytableau}
        {} & 1 \\
       {} & 0 & 2 & 3 \\
       {} & 0 & 1 & 2 & 3 \\
       {} & 1 & 3 \\
\end{ytableau}$
\hspace{1 cm}
$\lambda + 1 = 
\begin{ytableau}
       {} & 0 & 1 & 2 & 3 \\
       {} & 0 & 1 & 2 & 3 \\
       {} & 1 & 3 \\
\end{ytableau}$
\hspace{1 cm}
$\lambda^{\swarrow} = 
\begin{ytableau}
       {} & 2 & 3 \\
       {} & 0 & 1 & 2 & 3 \\
       {} & 1 & 3 \\
\end{ytableau}$
\end{center}
\end{example}


\subsection{Lattice Sequences and Lattice Paths}

Let $m$ be a nonnegative integer.  A \textit{lattice sequence $L = (L_0, L_1, L_2, \dots, L_{m}, L_{m+1})$} is a sequence of nonnegative integers satisfying $0=L_0 \leq L_1 \leq \dots \leq L_{m+1}$.  We say that $m$ is the \emph{width} of $L$ and that $L_{m+1}$ is the \emph{height} of $L$.  Note that if the width of $L$ is $m$, then $L$ has $(m+2)$ entries.

For $0 \leq i \leq j \leq m+1$, let $L_{[i, j]}$ denote the subsequence  $(L_i, L_{i+1}, \dots, L_j)$ and let $L_{[i, \dots]}$ be the subsequence $(L_i, L_{i+1}, \dots, L_{m+1})$. If $L$ and $L'$ are both sequences, let $L;L'$ denote the concatenation of the two sequences. Finally, for any sequence $L = (L_0, L_1, \dots)$, let $L + a$ denote the sequence $(L_0+a, L_1+a, \dots)$. 

A lattice sequence can be interpreted as a lattice path consisting of steps of the form $(1, 0)$ or $(0, 1)$ from $(0, 0)$ to $(m, L_{m+1})$.  For $1 \leq i \leq m$, the value of $L_i$ gives the height of the $ith$ horizontal step of the lattice path.  
Since lattice sequences can be interpreted as a lattice paths, the number of lattice sequences with width $m$ and height $n$ is $\binom{m+n}{m}$.  We will let $\mathcal{L}(m, n)$ denote the set of such lattice sequences. 

\begin{defn}
 Let $L \in \mathcal{L}(m, n)$.  The \emph{ascent set of $L$}, denoted $\asc(L)$, is given by
\[\asc(L) = \{i \in \{0, \dots, m\} \mid L_i < L_{i+1}\}.\]   
\end{defn}

When viewing $L$ as a lattice path, the ascent set of $L$ gives the horizontal coordinates in which vertical steps appear.

\begin{example}
    

The following is the lattice path corresponding to the lattice sequence $L = (0, 0, 0, 2, 3, 5) \in \mathcal{L}(4,5)$. The ascent set of $L$ is $\asc(L) = \{2, 3, 4\}$.
The entries in the sequence (excluding the first and last) correspond to the heights of the horizontal steps.

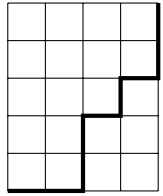
\begin{center}
\begin{tikzpicture}[scale = .5]
\draw (0,0) grid (4, 5);

\draw[ultra thick] (0,0) -- (2,0) -- (2, 2) -- (3, 2) -- (3, 3) -- (4, 3) -- (4, 5);

\end{tikzpicture}
\end{center}

\end{example}






We will now introduce some notation that will be heavily used in the proof of Theorem \ref{Big Boy}.  Let $L \in \mathcal{L}(m, n)$ and let $\eta \leq m$ be a nonnegative integer.  Let $\omega_\eta(L)$, $\gamma_{\eta}(L)$, and $\mu_{\eta}(L)$ be the following values:

\begin{itemize}[label = -]
\item $\omega_\eta(L) = L_{\eta+1}- L_{\eta},$

\item $\gamma_\eta(L) = \begin{cases}
\max(\asc(L_{[0, \eta]}) ), & \text{if } \asc(L_{[0, \eta]}) \neq \varnothing; \\
-1 & \text{ otherwise.}
\end{cases}, $
\item  $\mu_\eta(L) = \eta - 1 - \gamma_\eta(L).$
\end{itemize}

When viewing $L$ as a lattice path from $(0, 0)$ to $(m, n)$, $\omega_\eta(L)$ gives the number of vertical steps on the line $x = \eta$ and if $\gamma_\eta(L) \geq 0$, then $\gamma_\eta(L)$ gives the horizontal location of the east-most vertical step west of the line $x=\eta$.  If $\gamma_\eta(L) = -1$, then $L_{[0, \eta]} = (0, \dots, 0)$.  That is to say there are no vertical steps west of $x=\eta$.  The value of $\mu_\eta(L)$ gives the horizontal distance between the vertical step used to determine $\gamma_{\eta}(L)$ and line $x = \eta -1$.


\begin{example}
Consider the lattice sequence $L = (0, 0, 0, 0, 1, 2, 2, 2, 2, 2, 4, 4, 4, 4)$ where $\eta = 9$.  It follows that $\omega_{9}(L) = L_{10}-L_{9} = 4-2 = 2$, $\gamma_{9}(L) = 4$ and $\mu_9(L) = 9-1-4 = 4$.  The dashed blue line is the line $x = \eta$ and the dashed red line marks the location of the east-most vertical step west of $x=\eta$.

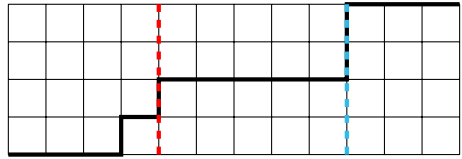
\begin{center}
\begin{tikzpicture}[scale = .5]
\draw (1,0) grid (13,4);

\draw[ultra thick] (1,0) -- (4, 0) -- (4, 1) -- (5,1) -- (5, 2) -- (10, 2) -- (10, 3) -- (10, 4) -- (13, 4);
\draw[ultra thick, dashed, color=lblue] (10,0) -- (10, 4);
\draw[ultra thick, dashed, color=red] (5,0) -- (5, 4);

\end{tikzpicture}
\end{center}

\end{example}





The following lemma will be used in the proof Theorem \ref{Big Boy}.

\begin{lemma}\label{Fixed Points}
Let $r$, $k$, $n$, and  $i$ be fixed nonnegative integers, and let $\eta = rk+n+1$. If
$A = \{L \in \mathcal{L}(\eta, k) \mid \mu_\eta(L) \geq r \cdot \omega_\eta(L) \text{ and } \omega_\eta(L) = k-i\}$, then \[|A| = \binom{(r+1)i+n}{i}.\]
\end{lemma}

\begin{proof}
We will prove this results by showing that there is a bijection between the sets $A$ and $\mathcal{L}(ri+n, i)$ as the size of $\mathcal{L}(ri+n, i)$ is $\binom{(r+1)i+n}{i}$.

Let $B = (B_1, \dots, B_{r(k-i)+1})$ be the sequence defined by $B_j = i$ if $j\leq r(k-i)$ and $B_{r(k-i)+1} = k$.  Let $\psi$ be a function with domain $\mathcal{L}(ri+n, i)$ defined by $\psi(L) = L;B$.  We will now show that $\psi(L) \in A$.

Let $L \in \mathcal{L}(ri+n, i)$.  Since $L_{ri+n+1} = i$ for all $L \in \mathcal{L}(ri+n, i)$, it follows that $L;B$ is weakly-increasing and hence is a lattice sequence with height $k$.  The number of elements in the lattice sequence $L;B$ is $(ri+n + 2) + (r(k-i)+1) = rk+n+3$. Since the width of a lattice sequence is the number of entries minus 2, the width of $L;B$ is $rk+n+1 = \eta$.  Note that $(L;B)_{\eta+1} - (L;B)_{\eta} = k - i$ and $\max(\asc(L;B)) = \max(\asc(L))\leq ri+n$.  This means that $\mu_\eta(L;B) \geq \eta -1 - (ri+n) = r(k-i) = r\omega_\eta(L;B)$.  Therefore we have that $\psi(L)$ is an element of $A$.

Now let $\phi$ be the function with domain $A$ defined by $\phi(L) = L_{[0, ri+n+1]}$.  We will now show that $\phi(L) \in \mathcal{L}(ri+n, i)$.  

Let $L \in A$.  This means that $\omega_\eta(L) = k-i$ and $\mu_\eta(L) \geq r \cdot (k-i)$.  Since $\omega_\eta(L) = k-i$, it follows that $L_\eta = k - (k-i) = i$ and because $\mu_\eta(L) \geq r \cdot (k-i)$, it follows that $L_j = i$ for all $j \in \{\eta-1-r\cdot(k-i)+1,\dots,\eta\}$.  In particular, $L_{ri+n+1} = i$.  Therefore $L_{[0, ri+n+1]} \in \mathcal{L}(ri+n, i)$ as desired.

Since $\psi$ and $\phi$ are inverses of each other (the sequence of terms removed by $\phi$ is the sequence $B$) we have that these sets are in bijection.
\end{proof}

\subsection{Labeled Lattice Sequences}

We will now turn our attention to the main combinatorial object of this article, the labeled lattice sequence.


\begin{defn}
    A \emph{labeled lattice sequence of type $(k, n, r)$} is a pair, $(\lambda, L)$, where $\lambda \in \Lambda(\alpha, r)$ for some $\alpha \in \mathcal{C}(k, r)$ and $L \in \mathcal{L}(rk+n+1+\ell(\alpha), k-|\alpha|)$.
\end{defn}

Notice that if $(\lambda, L)$ is a labeled lattice sequence, then the labels of $\lambda$ are independent of the constraints that $L$ satisfies.

\begin{prop}\label{Labeled Sequences}
Let $k$, $n$, and $r$ be nonnegative integers. 
Then the number of labeled lattice sequences, $(\lambda, L),$ of type $(k, n, r)$ such that $\sh(\lambda) = \alpha$ for a fixed $\alpha \in \mathcal{C}(k,r)$ is 
\[\Biggl(\prod_{j=1}^{\ell(\alpha)}\binom{r}{\alpha_j-1}\Biggr)\binom{(r+1)k+n+1-|\alpha|+\ell(\alpha)}{k-|\alpha|}.\]
\end{prop}

\begin{proof}
    If $r = 0$, then the empty composition is the only element of $\mathcal{C}(k, 0)$.  There are $\binom{k+n+1}{k}$ elements of $\mathcal{L}(n+1, k)$.

    Now assume $r \geq 1$.  By Proposition \ref{numberLT}, the number of labeled tableaux of shape $\alpha$ whose labels are all less than $r$ is $\Bigl(\prod_{j=1}^{\ell(\alpha)}\binom{r}{\alpha_j-1}\Bigr)$.  The number of $(rk+n+1+\ell(\alpha))$-width lattice sequences with height $(k-|\alpha|)$ is $\binom{(r+1)k+n+1+-|\alpha| + \ell(\alpha)}{k-|\alpha|}$.  Therefore the number of pairs is the product of these two numbers, as desired.
\end{proof}

\section{Generalized Hockey Stick Theorem}

We now have the necessary results and notation to state the main result of this article.

\begin{thm}\label{Big Boy}
Let  $k$, $n$, and $r$ be nonnegative integers and let $X$ be the set of labeled lattice sequences of type $(k,n,r)$. 
For each $(\lambda,L) \in X$ assign the weight $\wt((\lambda,L)) = sign(\lambda)$. Then, 
\begin{equation}\label{SRI}
    \sum_{(\lambda,L)\in X} \wt((\lambda,L)) = \sum_{i=0}^{k} \binom{(r+1)i+n}{i}.
\end{equation}
\end{thm}

\begin{proof}
    We will now define a sign-reversing involution $\tau \colon X \rightarrow X$ where the fixed points are enumerated by the right-hand side of Equation \ref{SRI}.  
    Let $(\lambda, L) \in X$ and 
    let $\eta = rk+n+1$, $\omega = \omega_{\eta}(L)$, $\gamma = \gamma_\eta(L)$, $\mu = \mu_{\eta}(L)$, and if $\lambda \neq ()$ let $\nu(\lambda, L) = \eta -1 - r \cdot \omega - \lambda_{1, 1}$.  We will denote $\nu(\lambda, L)$ by $\nu$.  
    

      Define $\tau \colon X \rightarrow X$ by $\tau(\lambda ,L) = (\lambda', L')$ where 
    
\[\lambda'  = \begin{cases}
\lambda \oplus (\mu \mod r)& \text{if } \mu < r\cdot\omega, \\ \\

\lambda & \text{if } \mu \geq r\cdot\omega  \\
 & \text{and }\sh(\lambda) = (), \\ \\

\lambda^{\swarrow}  & \text{if } \mu \geq  r\cdot\omega,
|\lambda_1| \geq 2, \\
& \text{and } \gamma \leq \nu,\\ \\

\lambda + (\mu \mod r)  & \text{if } \mu \geq r\cdot\omega, |\lambda_1| \geq 2, \\
& \text{and } \gamma > \nu.

\end{cases}\]
and 

\[L'  = \begin{cases}
L_{[0, \gamma]}; (L_{[\gamma+1, \eta]} -1 ); (L_\eta - 1 + \lfloor \frac{\mu}{r} \rfloor); (L_{[\eta+1, \dots]} - 2)& \text{if } \mu < r \cdot \omega, \\ \\

L & \text{if } \mu \geq r\cdot\omega  \\
 & \text{and }\sh(\lambda) = (), \\ \\

L_{[0, \nu]}; 
(L_{[\nu+1, \eta]} +1 );
(L_{[\eta+2, \dots]}+2)  
& \text{if } \mu \geq r\cdot\omega, \\
& |\lambda_1| = 2, \\
& \text{and } \gamma \leq \nu, \\ \\ 

L_{[0, \nu]}; 
(L_{[\nu+1, \dots]} +1 )  & \text{if } \mu \geq r\cdot\omega, \\
& |\lambda_1| > 2, \\
& \text{and } \gamma \leq \nu, \\ \\ 

 L_{[0, \gamma]}; (L_{[\gamma+1, \dots]} -1 )  & \text{if } \mu \geq r\cdot\omega, |\lambda_1| \geq 2, \\
& \text{and } \gamma > \nu.

\end{cases}\]


To show that $\tau$ is an involution, we need to establish that $(\lambda'', L'') = \tau(\lambda', L') = (\lambda, L)$, but to evaluate $\tau(\lambda', L')$, it may be necessary to find the values of $\omega_{\eta}(L')$, $\gamma_{\eta}(L')$, $\mu_{\eta}(L')$, and $\nu(L', \lambda')$. We will denote these values as $\omega'$, $\gamma'$, $\mu'$, and $\nu'$ respectively. We note that $\tau$ will be a sign reversing involution since by Proposition \ref{signofalpha}, $\sign(\lambda)=-\sign(\lambda')$ for all $\lambda$ where $\lambda \neq \lambda'$. Additionally, we note that $\sign(\lambda)=1$ for all $\lambda$ such that $\lambda=\lambda'$.

\noindent\textbf{Case 1: $\mu < r \cdot \omega$.}

Since $\mu < r \cdot \omega$, we have that $L' = L_{[0, \gamma]}; (L_{[\gamma+1, \eta]} -1 ); (L_\eta - 1 + \lfloor \frac{\mu}{r} \rfloor); (L_{[\eta+1, \dots]} - 2)$ and $\lambda' = \lambda \oplus (\mu \mod r)$.  This means that $|\lambda_1'| = 2$ and $\lambda'_{1, 1} = (\mu \mod r)$.  

By definition of $\omega'$, we have that 
\[\omega' = L'_{\eta+1} - L'_\eta = \Bigl(L_\eta - 1 + \Bigl\lfloor\frac{\mu}{r}\Bigr\rfloor\Bigr) - \Big(L_\eta -1\Bigr) = \Bigl\lfloor\frac{\mu}{r}\Bigr\rfloor.\]

It follows from the construction of $L'$ that the ascent set of $L'_{[0, \eta]}$ is a subset of the ascent set of $L_{[0, \eta]}$. Therefore it follows that $\gamma' \leq \gamma$ and similarly we conclude that $\mu' \geq \mu$.  This is because 

\[\mu' = \eta - 1 - \gamma' \geq \eta - 1 - \gamma = \mu. \]

Finally, the value of $\nu' = \gamma$ because 
\[\nu' = \eta -1 - r\cdot \omega' - \lambda'_{1, 1} = \eta - 1 - \Bigl(r \cdot \Bigl\lfloor\frac{\mu}{r}\Bigr\rfloor + (\mu \mod r)\Bigr) = \eta - 1 - \mu = \gamma.\]

Combining the equations and inequalities above, we have that 
$\mu' \geq \mu \geq r \cdot \bigl\lfloor\frac{\mu}{r}\bigr\rfloor = r \cdot \omega'$,
$|\lambda'_1| = 2$, and $\gamma' \leq \nu'$. 

Therefore we have  $\lambda'' = \lambda'^{\swarrow} = \lambda$ and
\begin{align*} 
L'' &=  L'_{[0, \nu']}; (L'_{[\nu'+1, \eta]} +1 );(L'_{[\eta+2,  \dots]}+2) \\ 
 &=  L_{[0, \gamma]}; (L_{[\gamma+1, \eta]} - 1 +1 ); (L_{[\eta+1, \dots]} - 2+2) \\
 & = L.
\end{align*}

\noindent\textbf{Case 2: $\mu \geq r \cdot \omega$ and $\sh(\lambda) = ()$.}

In this case, $\tau(\lambda, L) = (\lambda, L)$ so $\tau(\tau(\lambda, L)) = (\lambda, L)$.

\noindent\textbf{Case 3: $\mu \geq r \cdot \omega$, $|\lambda_1| = 2$, and $\gamma \leq \nu$.}

In this case, we have that 
$L' = 
L_{[0, \nu]}; 
(L_{[\nu+1, \eta]} +1 );
(L_{[\eta+2, \dots]}+2)$ and $\lambda' = \lambda^\swarrow$.

Therefore we have 
\[\omega' = L'_{\eta+1}-L'_\eta = L_{\eta + 2} + 2 -(L_\eta + 1) = L_{\eta + 2} - L_\eta + 1.\] 

Since $L$ is weakly increasing, we have that $L_{\eta + 2} - L_\eta + 1 \geq L_{\eta + 1} - L_\eta + 1 = \omega + 1> \omega$.  That is to say, $\omega' \geq \omega + 1$

It follows from the construction of $L'$ that $\nu \in \asc(L')$ and that $\nu$ is the greatest element of the ascent set of $L'_{[0, \eta]}$.  This is because $\gamma \leq \nu$ and $\gamma$ is the greatest element of the ascent set of $L_{[0, \eta]}$.  Therefore we have that $\gamma' = \nu$.  Since $\mu' = \eta - 1 - \gamma'$, we have that $\mu' = \eta - 1 - \nu$.  But since $\nu = \eta - 1- r \cdot \omega - \lambda_{1, 1}$, we have that $\mu' = r\cdot \omega + \lambda_{1, 1}$.  Note that $\lambda_{1, 1} < r$ so we have that $r\cdot \omega + \lambda_{1, 1} < r \cdot(\omega + 1)$.  But since $\omega' \geq \omega+1$, we have $\mu' < r \cdot \omega'$.


Since $\mu' < r \cdot \omega'$ and $\bigl\lfloor \frac{\mu'}{r}\bigr\rfloor = \omega$, we have that $\lambda'' = \lambda' \oplus (\mu' \mod r) = \lambda' \oplus \lambda_{1, 1} = \lambda$.
\begin{align*} 
L'' &=  L'_{[0, \gamma']}; \Bigl(L'_{[\gamma'+1, \eta]} -1 \Bigr);\Bigl(L'_\eta - 1 + \Bigl\lfloor\frac{\mu'}{r}\Bigr\rfloor\Bigr);\Bigl(L'_{[\eta+1, \dots]}-2\Bigr) \\ 
 &=  L_{[0, \nu]}; (L_{[\nu+1, \eta]} + 1 -1 );(L_\eta+1 - 1 + \omega);(L_{[\eta+2, \dots]} + 2 - 2) \\
 & = L.
\end{align*}


\noindent\textbf{Case 4: $\mu \geq r \cdot \omega$, $|\lambda_1| > 2$, and $\gamma \leq \nu$.}

In this case, we have that $L' = L_{[0, \nu]};(L_{[\nu+1,\dots]}+1)$ and $\lambda' = \lambda^{\swarrow}$.
As $|\lambda_1| > 2,$ the smallest entry of the first row of $\lambda$, $\lambda_{1, 1}$, is removed to create $\lambda'$. Therefore, $\lambda_{1,1} < \lambda'_{1,1}$ which implies $\nu > \nu'.$ 
By construction of $L'$, $\omega' = \omega$ and $\asc(L') = \asc(L) \cup \{\nu\}$.  As a result, $\gamma'=\nu$. But as $\nu > \nu',$ we have that $\gamma' > \nu'.$  We also have that that $\mu' \geq r \cdot \omega'$ and $(\mu' \mod r) = \lambda_{1, 1}$ as $\mu' = \eta - 1- \gamma' = \eta - 1- \nu = r \cdot\omega + \lambda_{1, 1}$.

Since $\mu'\geq r \cdot \omega'$ and $\gamma' > \nu'$, we have 
\begin{align*} 
L'' &=  L'_{[0, \gamma']}; (L'_{[\gamma'+1, \dots]} - 1) \\ 
 &=   L_{[0, \nu]}; (L_{[\nu+1, \dots]} + 1 - 1)\\
 & = L
\end{align*}
and $\lambda'' = \lambda' + (\mu' \mod r) = \lambda' + \lambda_{1, 1} = \lambda$.


\noindent \textbf{Case 5: $\mu \geq r \cdot \omega$, $|\lambda_1| \geq 2$ and $\gamma > \nu$.}

In this case, we must first establish that $\lambda+(\mu \mod r)$ is defined.
Recall that by definition, $0\leq \lambda_{1, 1} <r$.  Since $\gamma > \nu$, we have that $\eta - 1 - r \cdot \omega - \lambda_{1, 1} < \gamma$ or equivalently, $r\cdot\omega + \lambda_{1, 1} > \mu$.  Since $\mu \geq r\cdot\omega$ and $\lambda_{1, 1} < r$, we have that $0 \leq(\mu \mod r) < \lambda_{1, 1}$ and that $(\mu \mod r) \notin \lambda_{1}$ since $\lambda_{1, 1}$ is the smallest entry of $\lambda_1$.   
Since $\mu \geq r \cdot \omega$, $|\lambda_1| \geq 2$ and $\gamma > \nu$, we have that $\lambda'=\lambda + (\mu \mod r)$ and $L'= L_{[0,\gamma]};(L_{[\gamma+1, \dots]}-1)$. Note that this means that $|\lambda'_1| >2$. By construction of $L'$ we have that $\omega' = \omega$ and that $\gamma' \leq \gamma$. This implies that $\mu' \geq \mu \geq r \cdot \omega'$.


We will show that $\gamma' \leq \nu'.$ In particular, since $\gamma' \leq \gamma$, we will show $\nu'=\gamma$.  Recall that $\mu <r \cdot \omega + \lambda_{1, 1}$ and that $\lambda_{1, 1} < r.$  This means that $r \cdot \omega \leq \mu <r(\omega+1)$, or equivalently, $\mu=r \cdot \omega + (\mu \mod r)$. Since $\omega' = \omega$ and $(\mu \mod r) = \lambda'_{1, 1}$, we have that $\gamma = \eta - 1- \mu = \eta - 1 - r\cdot\omega' - \lambda'_{1, 1} = \nu'$.  Therefore $\gamma' \leq \nu'$.

Since $\mu' \geq r \cdot \omega'$, $|\lambda'_1| > 2$, and $\gamma'\leq \nu'$, we have 
\begin{align*} 
L'' &=  L'_{[0, \nu']}; (L'_{[\nu'+1, \dots]} + 1) \\ 
 &=   L_{[0, \gamma]}; (L_{[\gamma+1, \dots]} - 1 + 1)\\
 & = L
\end{align*}
and $\lambda'' = \lambda'^{\swarrow} = (\lambda + \lambda'_{1, 1} )^{\swarrow} = \lambda$.


Therefore the only fixed points of $\tau$ are when $\mu \geq r \cdot \omega$ and $\sh(\lambda) = ()$.  Since the lattice path of each fixed point is an element of $\mathcal{L}(rk+n+1, k)$ and each fixed point has the same shape $\lambda$ with $\sign(\lambda)=1$, then to complete the proof, we must show that $|\{L \in \mathcal{L}(rk+n+1, k) \mid \mu_\eta(L) \geq r \cdot \omega_\eta(L)\}| = \sum_{i=0}^{k} \binom{(r+1)i+n}{i}$. This follows immediately from Lemma \ref{Fixed Points} and the fact that the value of $\omega_\eta(L)$ can take on any value from $\{0, \dots, k\}$.
\end{proof}



When grouping together labeled tableaux of the same shape, Theorem \ref{Big Boy} can be restated in the following way.

\begin{cor}\label{Big Boy Cor}

For all nonnegative integers $k$, $n$, and $r$, we have
\begin{equation}\label{ZC}
    \sum_{i=0}^k\binom{(r+1)i+n}{i} = \sum_{\alpha \in \mathcal{C}(k, r)} \sign(\alpha) \Biggl(\prod_{j=1}^{\ell(\alpha)}\binom{r}{\alpha_j-1}\Biggr)\binom{(r+1)k+n+1-|\alpha|+\ell(\alpha)}{k-|\alpha|}.
\end{equation}
\end{cor}

\begin{proof}
This follows immediately from Proposition \ref{Labeled Sequences}. 
\end{proof}

When $r=0$, since the only element of $\mathcal{C}(k, 0)$ is the empty composition, Corollary \ref{Big Boy Cor} gives the Hockey Stick Identity (Equation \ref{Hockey Stick Identity}).

\begin{cor}
\[\sum_{i=0}^k\binom{i+n}{i} = \binom{k+1+n}{k}.\]
\end{cor}

When $r=1$, Corollary \ref{Big Boy Cor} gives the Big Hockey Stick and Pucks Identity (Equation \ref{Big Hockey Stick}). This is because the only compositions in $\mathcal{C}(k, 1)$ are those that consist only of 2's and there is a unique way to label each tableau.
\begin{cor}
\[\sum_{i=0}^k\binom{2i+n}{i} = \sum_{i=0}^{\lfloor\frac{k}{2}\rfloor}(-1)^i\binom{2k+1+n-i}{k-2i}.\]
\end{cor}

%

We conclude with an example to help illustrate which terms appear in the right-hand side of Equation \ref{ZC}.

\begin{example}
When $r=3$, $k=5$, and $n=0$, Equation \ref{ZC} says that 
\begin{equation}\label{ZC example eq}
   \sum_{i=0}^5\binom{4i}{i} =
\binom{21}{5}-3\binom{20}{3}+3\binom{19}{2}-\binom{18}{1}+3^2\binom{19}{1} -  3^2 \binom{18}{0} -  3^2 \binom{18}{0}. 
\end{equation}

Each term on the right-hand side of Equation \ref{ZC example eq} corresponds to one of compositions listed in Example \ref{comps}.  In Figure \ref{fig:GigaTriangle}, the terms from the left-hand side of Equation \ref{ZC example eq} are circled in brown and the terms from the right-hand side are circled in gray.

\end{example}

\begin{figure}
    \centering
\begin{tikzpicture}[scale = .75]
\foreach \i in {0, ..., 5}{
\draw[color=dbrown, ultra thick, fill=brown!35] (-\i,-4*\i) circle [radius=.4];
}
\draw[color=gray, ultra thick, fill = gray!35] (-5.5,-21) circle [radius=.4]; 
\draw[color=gray, ultra thick, fill = gray!35] (-7,-20) circle [radius=.4]; 
\draw[color=gray, ultra thick, fill = gray!35] (-7.5,-19) circle [radius=.4]; 
\draw[color=gray, ultra thick, fill = gray!35] (-8,-18) circle [radius=.4]; 
\draw[color=gray, ultra thick, fill = gray!35] (-8.5,-19) circle [radius=.4]; 
\draw[color=gray, ultra thick, fill = gray!35] (-9,-18) circle [radius=.4]; 

\foreach \n in {0,...,21} {
  \foreach \k in {0,...,\n} {
    \node at (\k-\n/2,-\n) {$\binom{\n}{\k}$};
    }
    }
\end{tikzpicture}
    \caption{Equation \ref{ZC} when $r=3$, $k=5$, and $n=0$.}
    \label{fig:GigaTriangle}
\end{figure}

\bibliography{references}{}

\begin{thebibliography}{1}

\bibitem{Bondarenko_1993}
Boris~A. Bondarenko.
\newblock Generalized {Pascal} triangles and pyramids their fractals graphs and
  applications.
\newblock {\em The Fibonacci Quarterly}, 31(2):162–164, May 1993.

\bibitem{Hilton_1987}
Peter Hilton and Jean Pedersen.
\newblock Looking into {Pascal}’s triangle: Combinatorics, arithmetic, and
  geometry.
\newblock {\em Mathematics Magazine}, 60(5):305--316, December 1987.

\bibitem{Jones_1996}
Charles~H. Jones.
\newblock Generalized hockey stick identities and $n$-dimensional blockwalking.
\newblock {\em The Fibonacci Quarterly}, 34(3):280–288, June 1996.

\bibitem{MehriARXIV}
Sima Mehri.
\newblock The hockey stick theorems in {Pascal} and trinomial triangles, 2016.
\newblock arXiv:1404.5106.

\bibitem{Merris_2003}
Russell Merris.
\newblock {\em Combinatorics}.
\newblock Wiley, August 2003.

\bibitem{OEIS}
{OEIS Foundation Inc.}
\newblock The {O}n-{L}ine {E}ncyclopedia of {I}nteger {S}equences, 2025.
\newblock Published electronically at \url{https://oeis.org}.

\bibitem{Stanley}
Richard~P. Stanley.
\newblock {\em Enumerative Combinatorics: Volume 1}.
\newblock Cambridge University Press, USA, 2nd edition, 2011.

\end{thebibliography}
\bibliographystyle{plain}

\noindent
(Concerned with sequences 
\href{https://oeis.org/A007318}{A007318},
\href{https://oeis.org/A000027}{A000027}, 
\href{https://oeis.org/A000217}{A000217},
\href{https://oeis.org/A000292}{A000292},
\href{https://oeis.org/A000332}{A000332},
\href{https://oeis.org/A000389}{A000389},
\href{https://oeis.org/A000579}{A000579},
\href{https://oeis.org/A006134}{A006134},
\href{https://oeis.org/A079309}{A079309},
\href{https://oeis.org/A057552}{A057552},
\href{https://oeis.org/A371965}{A371965},
\href{https://oeis.org/A371964}{A371964},
\href{https://oeis.org/A371963}{A371963},
\href{https://oeis.org/A188675}{A188675},
\href{https://oeis.org/A263134}{A263134},
\href{https://oeis.org/A087413}{A087413},
\href{https://oeis.org/A225612}{A225612}, and
\href{https://oeis.org/A225615}{A225615}.)

\end{document}